\documentclass[12pt,reqno]{amsart}
\usepackage{amssymb, amsmath,amsthm}
\usepackage{color}
\newtheorem{thm}{Theorem}

\numberwithin{defn}{section}
\numberwithin{thm}{section}
\numberwithin{Lemma}{section}
\numberwithin{Corollary}{section}
\numberwithin{Example}{section}
\numberwithin{subsection}{section}
\numberwithin{Remark}{section}
\numberwithin{equation}{section}
\numberwithin{ppn}{section}
\begin{document}
\title
[On Extending the Applicability of two-Step Secant Method...]
{On Extending the Applicability of two-Step Secant Method for non-differentiable operators} 
\author{Neha Gupta, J. P. Jaiswal}
\date{}
\maketitle

\textbf{Abstract.} 
The semi-local convergence analysis of a well defined and efficient two-step Secant method in Banach spaces is presented in this study. The recurrence relation technique is used under some weak assumptions. The pertinency of the assumed method is extended for nonlinear non-differentiable operators. The convergence theorem is also established to show the existence and uniqueness of the approximate solution. A numerical illustration is quoted to certify the theoretical part which shows that the earlier study will fail if the function is non-differentiable.


\section{\bf Introduction}
In diverse areas of science and engineering, there is an ample number of problems which can be seen in the form of
\begin{eqnarray}\label{eqn:11}
\pounds(a)=0,
\end{eqnarray}
which is mandatory to solve. Here, $\pounds:\Delta\subseteq A\rightarrow B$ be a continuous operator but non-differentiable.
Often, the solution of the equation $(\ref{eqn:11})$ cannot be found in the closed form. In this case, the iterative method is adopted to get the approximate solution.

An illustrious iterative method namely, Newton's method cannot be applied to solve the equation $(\ref{eqn:11})$ as the operator $\pounds$ is not differentiable
and hence in this situation, the Secant method can be chosen. There is plethora study of higher order method as it plays an important role where quick convergence is required like, applications where the stiff system of equations is involved. Moreover, many authors have studied the convergence analysis of various types of single-step iterations, multi-step iterations for $(\ref{eqn:11})$. In this manner, a well-known two-step King-Werner-type method having order $1+\sqrt{2}$ has been studied in the ref. [\cite{Werner}-\cite{Argyros}]. Initially, Werner in [\cite{Werner}-\cite{Werner1}] studied a method proposed by King in the article \cite{King} which is defined by:\\
Given $a_0,b_0\in \Delta$, Let
\begin{eqnarray}\label{eqn:12a}
a_{k+1}&=&a_k-\pounds'\bigg(\frac{a_k+b_k}{2}\bigg)^{-1}\pounds(a_k),\nonumber\\
b_{k+1}&=&a_{k+1}-\pounds'\bigg(\frac{a_k+b_k}{2}\bigg)^{-1}\pounds(a_{k+1}),
\end{eqnarray}
for each $k=0,1,2,\cdots$ and $\Delta\subseteq \mathbb{R}^n$
 In this continuation, McDougall et al. in \cite{Dougall} had studied a two-step method defined by:\\
For $a_0\in\Delta$,
\begin{eqnarray}\label{eqn:13}
b_0&=&a_0,\nonumber\\
a_1&=&a_0-\pounds'\bigg(\frac{a_0+b_0}{2}\bigg)^{-1}\pounds(a_0),\nonumber\\
b_k&=&a_k-\pounds'\bigg(\frac{a_{k-1}+b_{k-1}}{2}\bigg)^{-1}\pounds(a_k),\nonumber\\
a_{k+1}&=&a_k-\pounds'\bigg(\frac{a_k+b_k}{2}\bigg)^{-1}\pounds(a_k),
\end{eqnarray}
for each $k=1,2,\cdots$ and $\Delta\subseteq\mathbb{R}$
 On analyzing the equations $(\ref{eqn:13})$ and $(\ref{eqn:12a})$, one can notice that the method $(\ref{eqn:13})$ is simply the King-Werner-type method on using repeated initial points.
Method $(\ref{eqn:13})$ was also shown to be of order $1+\sqrt{2}$ in the ref. \cite{Dougall}. The convergence analysis can be categorized as local and semi-local which uses the details given at the solution and at the initial point, respectively. Here, we study the semi-local convergence analysis of the two-step Secant method which is more generalized and derivative-free. So, for $k=0,1,2,\cdots$ as
\begin{eqnarray}\label{eqn:12}
a_{k+1} &=& a_k-\Upsilon_k^{-1}\pounds(a_k),\nonumber\\
b_{k+1} &=& a_{k+1}-\Upsilon_k^{-1}\pounds(a_{k+1}),
\end{eqnarray}
where $a_0$ and $b_0$ are initial points, $\Upsilon_k=[a_k,b_k;\pounds]$. Here, $[a,b;\pounds]$ is a notation for a divided difference having order one for operator $\pounds$ which satisfies $[a,b;\pounds](a-b)=\pounds(a)-\pounds(b)$ for each $a,b\in \Delta$ with $a\not= b$. 
The local and semi-local convergence of the method $(\ref{eqn:13})$ has been established under various continuity conditions by using majorizing techniques which can be seen in the ref. [\cite{Ren1}-\cite{Lin}].

The interest in introducing the method $(\ref{eqn:12})$ is: the order of convergence of the method is similar to the method $(\ref{eqn:12a})$, the method $(\ref{eqn:12})$ is an appropriate substitute for the method $(\ref{eqn:12a})$, calculating $\pounds'(a)$ may be very expensive and hence the method $(\ref{eqn:12a})$ will be of no use. Hence, for all the above-mentioned statements, the aptness of the method $(\ref{eqn:12a})$ is extended through method $(\ref{eqn:12})$ and under weaker assumptions. 

In this article, we have two goals. First to assume a multi-parametric family of iterative methods which is derivative free. Next one is to get a semi-local convergence result for the nonlinear non-differentiable operators. Therefore, the following conditions are to be assumed:
\begin{flalign*}
&(A1) \|a_0-b_0\|\le s\ for\ a_0,b_0\in \Delta, &\\
&(A2)\|\Upsilon_0^{-1}\|\le\beta\ where\  \Upsilon_0^{-1}=[a_0,b_0;\pounds]^{-1},&\\
&(A3)\|\Upsilon_0^{-1}\pounds(a_0)\|\le\eta,&\\
&(A4)\|[a,b;\pounds]-[u,v;\pounds]\|\le\omega(\|a-u\|,\|b-v\|)\ \forall a,b,u,v\in \Delta,&
\end{flalign*}
where $s>0, \beta>0, \eta>0,\ \omega:\mathbb{R}_+\rightarrow\mathbb{R}_+$ is a continuous and non-decreasing function in its both arguments.
In the next section, we will corroborate the convergence theorem of the method $(\ref{eqn:12})$ for non-differentiable operators under weak continuity conditions.								



\section{\bf Semi-local Convergence Analysis of the Method $(\ref{eqn:12})$}
Given $a\in A$ and $\rho>0$, $R(a_0,\rho)$ will designate as an open ball around $a$ with radius $\rho$ and $\overline{R(a_0,\rho)}$ its closure. 

\begin{thm}\label{thm:21}
Let $\pounds:\Delta\subseteq A\rightarrow B$ be a nonlinear operator defined on a nonempty open convex domain $\Delta$ with two Banach spaces $A$ and $B$.
We assume that the conditions $(A1)-(A4)$ are satisfied and the following equation holds: 
\begin{equation}\label{eqn:21}
\mu\bigg(1-\frac{m}{1-\beta\omega(\mu,s+3\mu)}\bigg)-\eta=0,
\end{equation} 
where $m=max\{\beta\omega(\eta,s),(\beta\omega(\eta,\eta))\}$. The above equation has at least one positive root say, $\rho$ which is the smallest positive root of $(\ref{eqn:21})$. If $\beta\omega(\rho,s+3\rho)<1$, $W=\frac{m}{1-\beta\omega(\rho,s+3\rho)}<1$ and $\overline{R(a_0,\rho)}\subset\Delta$, then the sequence $\{a_k\}$ and $\{b_k\}$ produced by two-step Secant method $(\ref{eqn:12})$ converges to a unique solution $a^*$ of $\pounds(a)=0$. Moreover, the solution $a^*$ belongs to $R(a_0,\rho)$ and unique in $\overline{R(a_0,\rho)}$.
\end{thm}

\begin{proof}
Initially, by the virtue of mathematical induction we prove that the iterative sequence given in $(\ref{eqn:12})$ is well defined, that is, the iterative procedure is justifiable if the operator $\Upsilon_k$ is invertible and the point $a_{k+1},\ b_{k+1}$ lies in $\Delta$ at each step. From the initial hypotheses, it seems that $a_1$ is well defined and,
\begin{flalign*}
&\|a_1-a_0\|\le\|\Upsilon_0^{-1}\pounds(a_0)\|\le\eta\le \rho.
\end{flalign*}
Clearly, $a_1\in R(a_0,\rho)$. After that, we observe
\begin{eqnarray*}
\pounds(a_1)&=&\pounds(a_1)-\pounds(a_0)+\pounds(a_0)\\
&=&([a_1,a_0;\pounds]-[a_0,b_0;\pounds])(a_1-a_0)
\end{eqnarray*}
and
\begin{eqnarray*}
\|\pounds(a_1)\|&\le&\omega(\|a_1-a_0\|,\|a_0-b_0\|)\|a_1-a_0\|\\
 &\le&\omega(\rho,s+\rho)\eta.
\end{eqnarray*}
Consequently, we obtain
\begin{eqnarray*}
\|b_1-a_0\|&\le& \|a_1-a_0\|+\|\Upsilon_0^{-1}\pounds(a_1)\|\\
&\le&\eta+\beta\omega(\rho,s+\rho)\eta< \rho.
\end{eqnarray*}
Therefore $b_1\in R(a_0,\rho)$. From the second sub-step of the method $(\ref{eqn:12})$, we can get
\begin{eqnarray*}
\|b_1-a_1\| &\le& \beta\omega(\rho,s+\rho)\eta<\rho.
\end{eqnarray*}
Furthermore, we will show that $\Upsilon_1^{-1}$ exists and for this we have
\begin{eqnarray*}
\|I-\Upsilon_0^{-1}\Upsilon_1\| &\le& \|\Upsilon_0^{-1}\|\|\Upsilon_0-\Upsilon_1\|\\
 &\le&\beta\|[a_0,b_0;\pounds]-[a_1,b_1;\pounds]\|\\
&\le& \beta\omega(\rho,s+3\rho)<1.
\end{eqnarray*}
Hence, by using Banach Lemma, it follows that the operator $\Upsilon_1^{-1}$ exists and
\begin{eqnarray*}
\|\Upsilon_1^{-1}\|\le\frac{\beta}{1-\beta\omega(\rho,s+3\rho)}.
\end{eqnarray*}
Again, the approximation $a_2$ is well defined and
\begin{eqnarray*}
\|a_2-a_1\|&\le&\|\Upsilon_1^{-1}\pounds(a_1)\|\le \|\Upsilon_1^{-1}\|\|\pounds(a_1)\|\le W\eta\le\eta,\\
\|\pounds(a_2)\|&\le&\|[a_2,a_1;\pounds]-[a_1,b_1;\pounds]\|\|a_2-a_1\|\le\omega(\eta,\eta)\eta,\\
\|b_2-a_2\|&\le& \frac{\beta}{1-\beta\omega(\rho,s+3\rho)}\times \omega(\eta,\eta)W\eta<\rho.
\end{eqnarray*}
If we now suppose that $\Upsilon_j=[a_j,b_j;\pounds]$ is invertible and $b_{j+1},\ a_{j+1}\in R(a_0,\rho)\subseteq\Delta\ \forall\ j=1,2,3,\cdots,k-1$, then
\begin{eqnarray*}
&1)& \exists \  \Upsilon_j^{-1}=[a_j,b_j;\pounds]^{-1} \ such \ that \ \|\Upsilon_j^{-1}\|\le\frac{\beta}{1-\beta\omega(\rho,s+3\rho)},\\
&2)&\|a_{j+1}-a_j\|\le W\|a_j-a_{j-1}\|\le W^j \|a_1-a_0\|\le \eta,\\
&3)& \pounds(a_{j+1})\le \omega(\eta,\eta)\|a_{j+1}-a_j\|,\\
&4)&\|b_{j+1}-a_{j+1}\|\le W^j\eta.
\end{eqnarray*}
By induction hypotheses, we obtain
\begin{eqnarray*}
\|I-\Upsilon_0^{-1}\Upsilon_k\|&\le& \|\Upsilon_0^{-1}\|\|\Upsilon_0-\Upsilon_k\|\\
&\le&\|\Upsilon_0^{-1}\|\|[a_0,b_0;\pounds]-[a_k,b_k;\pounds]\|\\
&\le& \beta\omega(\rho,s+3\rho)<1.
\end{eqnarray*}
So by Banach lemma,
\begin{flalign*}
\|\Upsilon_k^{-1}\|\le\frac{\beta}{1-\beta\omega(\rho,s+3\rho)},
\end{flalign*}
Thus, we have
\begin{eqnarray*}
\|a_{k+1}-a_k\|&\le&\|\Upsilon_k^{-1}\|\|\pounds(a_k)\|\le W^k\eta<\eta.\\
\|a_{k+1}-a_0\|&\le& \|a_{k+1}-a_k\|+\|a_k-a_{k-1}\|+\cdots+\|a_1-a_0\|\\
&\le&(W^k+W^{k-1}+\cdots+1)\|a_1-a_0\|\\
&\le&\frac{1-W^{k+1}}{1-W}\|a_1-a_0\|<\rho.
\end{eqnarray*}
So, $a_{k+1}\in R(a_0,\rho)$. Subsequently,
\begin{eqnarray*}
\|\pounds(a_{k+1})\|&\le&\|[a_{k+1},a_k;\pounds]-[a_k,b_k;\pounds]\|\|a_{k+1}-a_k\|\\
&\le& \omega(\|a_{k+1}-a_k\|,\|a_k-b_k\|)\|a_{k+1}-a_k\|\\
&\le&\omega(\eta,\eta)\|a_{k+1}-a_k\|
\end{eqnarray*}
\begin{flalign*}
&\Rightarrow\|b_{k+1}-a_{k+1}\|\le \|\Upsilon_k^{-1}\pounds(a_{k+1})\|\le W^{k+1}\eta.&
\end{flalign*}
Besides this, we will show that $b_{k+1}\in R(a_0,\rho)$.
\begin{eqnarray*}
\|b_{k+1}-a_0\|&\le& \|b_{k+1}-a_{k+1}\|+\|a_{k+1}-a_k\|+\cdots+\|a_1-a_0\|\\
&\le&(W^{k+1}+W^k+W^{k-1}+\cdots+1)\|a_1-a_0\|\\
&\le&\frac{1-W^{k+2}}{1-W}\|a_1-a_0\|<\rho
\end{eqnarray*}
$\Rightarrow b_{k+1}\in R(a_0,\rho)$. Hence, the mathematical induction is true for all $j=1,2,3\cdots n$. Eventually we will show that the sequence $\{b_k\}$ is a Cauchy sequence. For this, let $p\ge 1$,
\begin{eqnarray*}
\|b_{k+p}-b_k\|&\le& \|b_{k+p}-a_{k+p}\|+\|a_{k+p}-a_{k+p-1}\|+\cdots+\|a_{k+1}-a_k\|\\
&\le& (W^p+W^{p-1}+W^{p-2}+\cdots+1)\|a_{k+1}-a_k\|\\
&\le& \frac{1-W^p}{1-W}W^k\|a_1-a_0\|\le\frac{W^k\eta}{1-W}.
\end{eqnarray*}
Since, $W<1$, $\{b_k\}$ is a Cauchy sequence. In a similar manner, we can show that $\{a_k\}$ is a Cauchy sequence. Thus, the sequence $\{a_k\}$ and $\{b_k\}$ are convergent and converges to $a^*\in R(a_0,\rho)$.\\
To claim uniqueness of the solution, let $\exists$ another solution $b^*$ of $\pounds(a)=0$ in $R(a_0,\rho)$ such that $\pounds(b^*)=0.$ Consider the operator, $T=[a^*,b^*;\pounds]$ and if $T$ is invertible then $a^*=b^*$. Now let,
\begin{eqnarray*}
\|\Upsilon_0^{-1}T-I\|&\le& \|T-\Upsilon_0\|\\
&\le& \|\Upsilon_0^{-1}\|\|[a^*,b^*;\pounds]-[a_0,b_0;\pounds]\|\\
&\le& \beta\omega(\|a^*-a_0\|,\|b^*-b_0\|)\\
&\le& \beta\omega (\rho,s+3\rho)<1.
\end{eqnarray*}
Hence, the operator $T^{-1}$ exists by Banach lemma and $a^*=b^*.$
\end{proof}



\section{\bf Numerical Example}
\textbf {Example 3.1}\cite{Ren}
Let $A=B=\Delta=\mathbb{R}^2$.
Consider an operator $\pounds=(\pounds_1,\pounds_2)$ on $\Delta$ by
\begin{eqnarray}
\pounds_1(a_1,a_2)&=&a_1^2-a_2+1+\frac{1}{9}|a_1-1|,\nonumber\\
\pounds_2(a_1,a_2)&=&a_2^2+a_1-7+\frac{1}{9}|a_2|,\nonumber
\end{eqnarray}
where $a=(a_1,a_2)\in \mathbb{R}^2$ and we use infinity norm here.
For $v,w\in\mathbb{R}^2$, we take $[v,w;\pounds]\in L(A,B)$ as\\
\begin{eqnarray*}
[v,w;\pounds]_{i1}=\frac{\pounds_i(v_1,w_2)-\pounds_i(w_1,w_2)}{v_1-w_1} \ , [v,w;\pounds]_{i2}=\frac{\pounds_i(v_1,v_2)-\pounds_i(v_1,w_2)}{v_2-w_2}.
\end{eqnarray*}
Therefore,
\begin{eqnarray*}
[v,w;\pounds]=
\begin{pmatrix}
\frac{v_1^2-w_1^2}{v_1-w_1}&&-1\\
1&& \frac{v_2^2-w_2^2}{v_2-w_2}
\end{pmatrix}
+\frac{1}{9}
\begin{pmatrix}
\frac{|v_1-1|-|w_1-1|}{v_1-w_1}&&0\\
0&&\frac{|v_2|-|w_2|}{v_2-w_2}
\end{pmatrix},
\end{eqnarray*}
and,
\begin{eqnarray*}
\|[a,b;\pounds]-[v,w;\pounds]\|\le \|a-v\|+\|b-w\|+\frac{2}{9}.
\end{eqnarray*}
So, we can take $\omega(a,b)=a+b+\frac{2}{9}$. Clearly, here the conditions assumed in \cite{Abhimanyu} fails as the function is non-differentiable. Here,
we choose $a_0=(1.06,2.40)\ , b_0=(1.14,2.54)$. For Theorem $(\ref{eqn:21})$, we can obtain the following parameters:\\
$a_1\approx (1.1607,2.3629),\ b_1\approx(1.1593,2.3619),\ \beta\approx 0.4775,\ s\approx 0.14,\\
\rho\approx 0.1997, \eta\approx 0.1007$ and $m\approx0.2210$. In this case, the solution of equation $(\ref{eqn:21})$ are satisfied which confirms that the unique solution of $\pounds(a)=0$ exists in $\overline{R(a_0,\rho)}$. As a solution of the equation $(\ref{eqn:11})$ we acquire the vector $a^*\approx(1.159361,2.361824)$ after second iterations.


\section{Conclusion}
In this work, we scrutinize the semi-local convergence result of the two-step Secant method when applied for non-differentiable operators. In this idea, basically we have extended the results of Kumar et al. \cite{Abhimanyu} where the author has considered the applicability of the method for differentiable case only. Hence, it is noteworthy that we have extended the implications of the two-step Secant method for non-differentiable operators. A concrete example is also considered to sustain the theory.


\section{Acknowledgment}
This work is supported by Science and Engineering Research Board (SERB), New Delhi, India under the scheme Start-up-Grant (Young Scientists) (Ref. No. YSS/2015/001507).
                  

Neha Gupta\\
Department of Mathematics\\
Maulana Azad National Institute of Technology\\
 Bhopal, M.P. India-462003.\\
Email: neha.gupta.mh@gmail.com.\\\\
Jai Prakash Jaiswal\\
Department of Mathematics\\
Maulana Azad National Institute of Technology \\
Bhopal, M.P. India-462003. \\
Email: asstprofjpmanit@gmail.com.

\end{document}